\documentclass[11pt]{amsart}
\usepackage{amscd}
\usepackage{pstricks}
\usepackage{color}
\theoremstyle{plain}
\newtheorem{thm}{Theorem}[section]
\newtheorem{theorem}[thm]{Theorem}

\newtheorem{proposition}[thm]{Proposition}
\theoremstyle{definition}
\newtheorem{remark}[thm]{Remark}

\newtheorem{question}[thm]{Question}

\newtheorem{setup}[thm]{Set-up}
\numberwithin{equation}{section}

\newcommand{\C}{{\mathbb C}}

\renewcommand{\P}{{\mathbb P}}
\newcommand{\Q}{{\mathbb Q}}
\newcommand{\R}{{\mathbb R}}

\newcommand{\Z}{{\mathbb Z}}

\title [On a question of Professor Igor Dolgachev]{A few explicit examples of complex dynamics of inertia groups on surfaces - a question of Professor Igor Dolgachev}
\author{Keiji Oguiso}
\address{Mathematical Sciences, the University of Tokyo, Meguro Komaba 3-8-1, Tokyo, Japan and Korea Institute for Advanced Study, Hoegiro 87, Seoul, 
133-722, Korea}
\email{oguiso@ms.u-tokyo.ac.jp}
\thanks{The author is supported by JSPS Grant-in-Aid (S) No 25220701, JSPS Grant-in-Aid (S) 15H05738, JSPS Grant-in-Aid (B) 15H03611, and by KIAS Scholar Program.}
 \dedicatory{Dedicated to Professor JongHae Keum on the occasion of his
sixtieth birthday.}

\begin{document}

\maketitle

\begin{abstract}
We give a few explicit examples which answer an open minded question of Professor Igor Dolgachev on complex dynamics of the inertia group of a smooth rational curve on a projective K3 surface and its variants for a rational surface and a non-projective K3 surface. 
\end{abstract}

\section{Introduction}

In this note we work over $\C$. Our main results are Theorems \ref{thm1}, \ref{thm2}, \ref{cor1}, \ref{thm3}.

Let $V$ be a smooth projective variety and $C$ a subvariety of $V$. We define the subgroups of ${\rm Aut}\, (V)$ called the {\it decomposition group} of $C$ and the {\it inertia group} of $C$, respectively by
$${\rm Dec}\,(C) := \{f \in {\rm Aut}\, (V)\, |\, f(C) = C\}\,\, ,\,\, {\rm Ine}\,(C) := \{f \in {\rm Dec}\, (V)\, |\, f|_{C} = id_C\}\,\, .$$ 
The aim of this note is to add some answers to the following open minded question and its variants. The question has been asked by Professor Igor Dolgachev in his openning talk at the conference "Algebraic Geometry in honnor of Professor JongHae Keum's 60-th birthday":

\begin{question}\label{ques1}
It is challenging to find (further) pairs $(S, C)$ of a projective K3 surface and a smooth rational curve $C \subset S$ such that ${\rm Ine}\,(C)$ contains an element with positive topological entropy. 
\end{question}

Though it is not the definition, the {\it topological entropy} $h_{{\rm top}}(f)$ of an automorphism $f \in {\rm Aut}\, (V)$ is given by:
$$h_{{\rm top}}(f) = \log \rho(f) \ge 0\,\, .$$
Here $\rho(f)$ is the spectral radius of the action of $f$ on the group of the numerical equivalence classes of algebraic cocycles $\oplus_{p=0}^{\dim\, V} N^p(V)$ when $V$ is projective. More generally, $\rho (f)$ is the spectral radius of the action of $f$ on $\oplus_{p=0}^{\dim\, V}H^{p, p}(V)$, and these two quantities coincides if $V$ is projective. Originally, $h_{{\rm top}}(f)$ is defined as a fundamental measure "how fast forward orbits $\{f^n(x) | n \ge 1\}$, $\{f^n(y) | n \ge 1\}$ of two general points $x, y \in V$ spread out" and therefore it is considered as a fundamental measure of the complexity of complex dynamics given by the forward iterations of $f$ (see \cite{Gr87}, \cite{KH95}, \cite{DS05}, \cite{Tr15} for more details). We have $h_{{\rm top}}(f) = 0$ if it is tame and $h_{{\rm top}}(f) > 0$ if it is wild. So, Question \ref{ques1} asks two extreme aspects; complex dynamics of $f$ on $S \setminus C$ have to be rather wild, while it is trivial on $C$. In a deeper level, the equidistribution of {\it isolated} fixed sets $Q_n$ of $f^n$ under $n \to \infty$ for any surface automorphism $f$ of positive entropy is discovered in the survey \cite[Theorem 5.6]{DS16}. The difficulty in the proof is caused by fixed curves. Pairs $(S, C)$ in Question \ref{ques1} provide explicit examples of equidistribution theorem with fixed curves. This is also an interesting dynamical feature of Question \ref{ques1}. We remark that the higher dimensional case is an open problem (see eg. \cite{DS16}). It is also natural to study the action of $f$ on the normal bundle $N_{S|C}$. In our construction, it is trivial in Theorems \ref{thm1} \ref{thm2} \ref{cor1} (projective cases) and it is the multiplication by a Salem number in Theorem \ref{thm3} (non-projective case). 

Following \cite{DZ01}, we call a smooth rational surface $W$ a {\it Coble surface} if $|-K_W| = \emptyset$ and $|-2K_W| \not= \emptyset$. Let $B \subset \P^2$ be a sextic plane curve with $10$ nodes, $T$ the blow up of $\P^2$ at the $10$ nodes and $C \subset T$ the proper transform of $B$. Then $\P^1 \simeq C \in |-2K_T|$ and $T$ is a Coble surface called {\it classical} (\cite{DZ01}). $T$ admits a finite double cover $\pi : S \to T$ branched along $C$. Then $S$ is a projective K3 surface, which we call {\it of classical Coble type}. We denote the ramification curve $(\pi^*C)_{\rm red} \subset S$ by the same letter $C$. In his talk, Professor Igor Dolgachev, among other things, showed that the pair $(S, C)$ of a K3 surface $S$ of classical Coble type and its ramification curve $C$ gives an affirmative answer to Question \ref{ques1}. He shows first that $(T, C)$ satisfies a similar property and then lifts to $(S, C)$. 

The aim of this note is to remark the following theorems (Theorems \ref{thm1}, \ref{thm2}, \ref{cor1}, \ref{thm3}): 

\begin{theorem}\label{thm1}
Question \ref{ques1} is affirmative for a smooth Kummer surface ${\rm Km}\, (E\times F)$ associated to the product abelian surface of elliptic curves $E$ and $F$, i.e., there is a smooth rational curve $C \subset {\rm Km}\, (E\times F)$ such that ${\rm Ine}\, (C)$ has an element of positive topological entropy.  
If in addition $E$ and $F$ are not isogenous, then there is a smooth rational curve $C \subset {\rm Km}\, (E\times F)$ satisfying the following two properties:

\begin{enumerate}
\item $[{\rm Aut}\,({\rm Km}\, (E\times F)) : {\rm Dec}\,(C)] < \infty$ and $[{\rm Dec}\,(C) : {\rm Ine}\, (C)] = \infty$ and 
\item There is $f \in {\rm Ine}\, (C)$ with $h_{{\rm top}}(f) > 0$. 
\end{enumerate}
\end{theorem}
The properties (1) and (2) claim that ${\rm Dec}\, (C)$ is almost the same as ${\rm Aut}\, (S)$ and ${\rm Ine}\, (C)$ is much smaller than ${\rm Dec}\, (C)$ but ${\rm Ine}\, (C)$ yet enjoys rich dynamics as explained above. Note that if $E$ and $F$ are not isogenous, then ${\rm Km}\, (E\times F)$ is a $2$-elementary K3 surface in the sense of \cite{Ni81}, and at the same time, it is a finite double cover of a non-classical Coble surface in the sense of \cite{DZ01}. So, it is very close to Professor Igor Dolgachev's example, even though our proof is completely different. We prove Theorem \ref{thm1} in Section \ref{sect3}.

Let $\zeta_3 = (-1 + \sqrt{-3})/2 \in \C$ be a primitive third root of unity and $E_{3}$ the elliptic curve of period $\zeta_3$. Then $E_3 \times E_3$ has an automorphism $\tau_3$ of order three defined by 
$$\tau_3^*(x, y) =(\zeta_3 x, \zeta_3^2y)\,\, .$$
Here $(x, y)$ are the standard coordinates of the universal covering space $\C \times \C$ of $E_3 \times E_3$. The minimal resolution $S_3$ of the quotient surface $(E_3 \times E_3)/\langle \tau_3 \rangle$ is a projective K3 surface called (one of) the most algebraic K3 surfaces (\cite{Vi83}). 

\begin{theorem}\label{thm2}
Question \ref{ques1} is affirmative for the most algebraic K3 surface $S_3$. More precisely, there is a smooth rational curve $C \subset S_3$ satisfying the following two perperties:

\begin{enumerate}
\item $[{\rm Aut}\,(S_3) : {\rm Dec}\,(C)] < \infty$ and $[{\rm Dec}\,(C) : {\rm Ine}\, (C)] = \infty$ and 
\item There is $f \in {\rm Ine}\, (C)$ with $h_{{\rm top}}(f) > 0$. 
\end{enumerate}
\end{theorem}

As an application of Theorem \ref{thm2}, we also show the following:

\begin{theorem}\label{cor1}
There are a smooth rational surface $W$ and a smooth rational curve $C \subset W$ such that 
\begin{enumerate}
\item $[{\rm Aut}\,(W) : {\rm Dec}\,(C)] < \infty$ and $[{\rm Dec}\,(C) : {\rm Ine}\, (C)] = \infty$;
\item There is $f \in {\rm Ine}\, (C)$ with $h_{{\rm top}}(f) > 0$; and 
\item $|mK_W| = \emptyset$ for $m = -1$ and $-2$, while $|-3K_W|$ consists of a unique element. In particular, $W$ is not isomorphic to any Coble surface. 
\end{enumerate} 
\end{theorem}

We construct $W$ as the minimal resolution of a suitable quotient of $S_3$. Compare Theorem \ref{cor1} with results due to McMullen \cite{Mc07}, rich rational surface dynamics arizing from the decomposition group of the aniti-canonical cuspidal curve and Bedford-Kim \cite{BK09}, rich rational surface dynamics with no stable curve. In Section \ref{sect4}, we prove Theorems \ref{thm2} and derive Theorem \ref{cor1} from Theorem \ref{thm2}. 

As a referee asked, it is interesting to see "how large the subset $\{f \in {\rm Ine}\, (C)\,\, |\,\, h_{{\rm top}}(f) > 0\}$ of ${\rm Ine}\, (C)$ is" for each of our constructions above. The author feels that the subset should be "quite large" but he has no good idea to formulate precisely. 

\begin{theorem}\label{thm3}
Question \ref{ques1} is affirmative for some K3 surface $S$ with no non constant global meromorphic function, i.e., of algebraic dimension $0$. More precisely, there are a K3 surface of algebraic dimension $0$ and smooth rational curves $C \subset S$ and $D \subset S$ satisfying the following perperties:

\begin{enumerate}
\item ${\rm Aut}\,(S) = {\rm Dec}\,(C) = {\rm Ine}\, (C)$, and there is $f \in {\rm Ine}\, (C)$ with $h_{{\rm top}}(f) > 0$;
\item ${\rm Aut}\,(S) = {\rm Dec}\,(D)$ and ${\rm Ine}\, (D) = \{id_S\}$.
\end{enumerate}
\end{theorem}

As a candidate $S$ in Theorem \ref{thm3}, we choose a K3 surface constructed in \cite{Og10}. Note that $\{f \in {\rm Ine}\, (C)\,\, |\,\, h_{{\rm top}}(f) > 0\} = {\rm Ine}\, (C) \setminus \{id_S\}$ in (2), as ${\rm Aut}\, (S) \simeq \Z$ (\ref{Og10}). Our proof is based on a very special feature occuring only on a non-projective K3 surface. We prove Theorem \ref{thm3} in Section \ref{sect5}. 

In Question \ref{ques1}, there are two essential issues to consider; "How to find an automorphisms of positive entropy" and "How to find a candidate curve $C$". Our main approach in Theorems \ref{thm1}, \ref{thm2} is to seek two different elliptic fibrations with positive Mordell-Weil rank. In Section \ref{sect2}, we briefly recall an existing method (Proposition \ref{prop21}) and synthesize it with the inertia group of some smooth rational curve in fibers (Proposition \ref{prop22}). Proposition \ref{prop22} is the main tool in our proof of Theorems \ref{thm1}, \ref{thm2}. For Theorem \ref{thm3}, we use a non-symplectic automorphism of a K3 surface whose character of the action on the space of the global holomorphic $2$-forms is not root of unity but a Galois conjugate of a Salem number, 
{\it a very special property occuring only on non-projective K3 surfaces}. It is worth recalling that this property also played an essential role in constructing automorphisms of (necessarily non-projective) K3 surfaces with Siegel disk by McMullen (\cite{Mc02}, see also \cite{Og10}). 

{\bf Acknowledgements.} First of all, I would like to express my thanks to Professor Igor Dolgachev for his inspiring talk, interest in this work and valuable discussions. I would like to thank Professors Viacheslav Nikulin, Yuya Matsumoto, Shigeru Mukai, Matthias Schuett, Nessim Sibony, Xun Yu and referees for valuable suggestions, comments and careful reading and Professor JongHae Keum for his interest in this work. I would like to express my thanks to the organizers of Professor JongHae Keum's 60th birthday conference for invitation.

\section{Some basic terminologies for lattices and K3 surfaces}\label{sect01} 

In this section, we briefly recall basic terminologies concerning lattices and K3 surfaces, which are used throughout this note. 

We call a pair $(L, (*, **))$ of a free $\Z$-module $L$ of positive finite rank $r$ and a symmetric bilinear form $(*, **) : L \times L \to \Z$ {\it a lattice of rank} $r$. We often write the lattice $(L, (*, **))$ simply by $L$ when no confusion arizes. For any ring $A$, we denote by $L_{A}$ the $A$-module $L \otimes_{\Z} A$. We call the lattice $L = (L, (*, **))$ is hyperbolic (resp. negative definite) if the signature of the bilinear form $(*, **)$ is $(1, r-1)$ and $r \ge 2$ (resp. $(0, r)$ and $r \ge 1$), considered as the real symmetric bilinear form on $L_{\R}$. We call a lattice $L$ non-degenerate, if there is no $x \in L \setminus \{0\}$ such that $(x, y) = 0$ for all $y \in L$. 

Let $L$ be a hyperbolic lattice. Under the Eucldean topology of $L_{\R}$, the subset $\tilde{P}(L) := \{x \in L_{\R}\, |\, (x, x) > 0\}$ has two connected components. We choose and fix ${P}(L)$, which is any one of the two connected components of $\tilde{P}(L)$, and call it the positive cone. We denote by ${\rm O}(L)$ the orthogonal group of a lattice $L$
$${\rm O}(L) := \{f \in {\rm GL}\, (L)\,|\, (f(x), f(y)) = (x, y)\,\, \forall x, y \in L\}$$ and by ${\rm O}^{+}(L)$ the subgroup of ${\rm O}\, (L)$ consisting $f \in {\rm O}\, (L)$ such that $f_{\R} := f \otimes id_{\R} \in {\rm GL}\, (L_{\R})$ preserves the positive cone $P(L)$. We denote the boundary of $P(L)$ in $L_{\R}$ by $\partial P(L)$. Then $f_{\R}(\partial P(L)) = \partial P(L)$ if $f \in {\rm O}^{+}(L)$. 

Let $S$ be a K3 surface. We denote a nowhere vanishing global holomorphic $2$-form on $S$ by $\omega_S$. By definition of K3 surface, $\omega_S$ is unique up to $\C^{\times}$. Important lattices in this note are the second cohomology lattice $H^2(S, \Z)$, the N\'eron-Severi lattice ${\rm NS}\, (S)$ and the transcendental lattice $T(S) := {\rm NS}\, (S)^{\perp}$ in $H^2(S, \Z)$ for a K3 surface $S$. The symmetric bilinear form on these lattices are the intersection form. $T(S)$ is also the minimal primitive sublattice of $H^2(S, \Z)$ such that $\C\omega_S \subset T(S)_{\C}$. Note also that if ${\rm NS}\, (S)$ is non-degenerate, then ${\rm NS}\, (S) \oplus T(S) \subset H^2(S, \Z)$ and $[H^2(S, \Z) : {\rm NS}\, (S) \oplus T(S)] < \infty$. We refer to \cite{BHPV04} for basic facts on K3 surfaces and surfaces. 

\section{Complex dynamics arizing from two different elliptic fibrations}\label{sect2} 

Our main result of this section is Proposition \ref{prop22}. 

Let $X$ be a projective K3 surface. We call a surjective morphism $\varphi : X \to \P^1$ an elliptic fibration if the generic fiber $X_{\eta}$ is an elliptic curve defined over the function field $\C(\P^1)$ of the base. We always choose the orgin $O$ of $X_{\eta}$ in $X_{\eta}(\C(\P^1))$. The set $X_{\eta}(\C(\P^1))$ then forms an additive group ${\rm MW}\,(\varphi)$ with unit $O$, called the Mordell-Weil group of $\varphi$. As $\varphi : X \to \P^1$ is relatively minimal, the birational automorphism of $X$ given by the translation on $X_{\eta}$ by $P \in {\rm MW}\,(\varphi)$ is a biregular automorphism of $X$, i.e., ${\rm MW}\,(\varphi) \subset {\rm Aut}\, (X)$. By definition, ${\rm MW}\, (\varphi)$ is an abelian group and coincides with the group of translations of global sections of $\varphi$, under the natural bijective correspondence between $X_{\eta}(\C(\P^1))$ and the set of global sections of $\varphi$. 

In this section, we are interested in a projective K3 surface in the following:
\begin{setup}\label{st21}
$X$ is a projective K3 surface admitting two different elliptic fibrations $\varphi_i : X \to \P^1$ ($i =1$, $2$) whose Mordell-Weil group ${\rm MW}(\varphi_i)$ has an element $f_i$ {\it of infinite order} for each $i = 1$, $2$. 
\end{setup}
Here two elliptic fibrations $\varphi_i$ ($i=1$, $2$) are said to be different if the generic point of the generic fiber, in the sense of scheme, are different (non-closed) points of $X$. 
 
We say that a property (P) (concerning closed points of $X$) holds for very general points if there is a countable union $B$ of proper closed subvarieties of $X$ such that the property (P) holds for any closed point $x \in X \setminus B$. 

K3 surfaces in Set-up \ref{st21} have rich complex dynamical properties, as the next proposition shows: 
\begin{proposition}\label{prop21} Under Set-up \ref{st21}, the following holds: 

\begin{enumerate}
\item The orbit $\langle f_1, f_2 \rangle \cdot x$ is dense in $X$ for very general points $x \in X$. Here topology of $X$ is not the Zariski topology but the Euclidean topology  (cf. \cite[Questions, Page 206]{Mc02} for a relevant open question). 

\item There is a positive integer $n$ such that $\langle f_1^n, f_2^n \rangle = \langle f_1^n \rangle * \langle f_2^n \rangle \simeq \Z * \Z$ (the free product). 

\item There is $f \in \langle f_1, f_2 \rangle$ such that $h_{{\rm top}}(f) > 0$.
\end{enumerate}
\end{proposition}

\begin{proof} The assertion (1) is observed by \cite{Ca01}. The assertion (2) is proved by \cite{Og08} in a more general setting and is also related to Tits' aternative for the automorphism group of a compact K\"ahler manifold (\cite{Zh09} and references therein). 

The assertion (3) is implicit in \cite{Og07}. As the assertion (3) is crucial in this note, we shall give a complete proof here. 

Let $v_i \in {\rm NS}\, (X)$ be the fiber class of $\varphi_i$ ($i=1$, $2$). Then $v_i$ ($i=1$, $2$) are primitive integral nef classes and they are $\Q$-linearly independent in ${\rm NS}\,(X)_{\Q}$ as $\varphi_i$ are different. 

Now assuming to the contrary that $h_{{\rm top}}(f) = 0$ for all $f \in \langle f_1, f_2 \rangle$, we shall derive a contradiction. We use the following theorem proved in \cite{Og07}:

\begin{theorem}\label{thm21} Let $L$ be a hyperbolic lattice and $G < {\rm O}^+\, (L)$ a subgroup of ${\rm O}^+\, (L)$. Assume that the natural logarithm of spectral radius of $g$ is $0$ for all $g \in G$. Then, there are a subgroup $N$ of $G$ such that $[G:N] < \infty$ and a primitive integral element $v \in L \cap \partial P(L) \setminus \{0\}$ such that $g(v) = v$ for all 
$g \in N$. 
\end{theorem}

As $X$ is a smooth projective surface, ${\rm NS}\, (X)$ is a hyperbolic lattice of signature $(1, \rho(X) -1)$. Here $\rho(X)$ is the Picard number of $X$ and $\rho(X) \ge 2$ as an ample class of $X$ and the fiber class $v_1$ are $\Q$-linearly independent in ${\rm NS}\, (X)_{\Q}$. As usual, we choose the positive cone $P({\rm NS}\, (X))$ so that $P({\rm NS}\, (X))$ contains the ample cone of $X$. Then $\langle f_1, f_2 \rangle$ preserves the positive cone, i.e., $\langle f^*_1, f^*_2 \rangle < {\rm O}^{+}({\rm NS}\, (X))$. Here $f^*$ is the natural action of $f$ on ${\rm NS}\, (X)$. As we assume that $h_{{\rm top}}(f) = 0$ for all $f \in \langle f_1, f_2 \rangle$, i.e., the natural logarithm of the spectral radius of $f^*_{\C}$ is $0$ for all $f^* \in \langle f^*_1, f^*_2 \rangle$, it follows from Theorem \ref{thm21} that there is a finite index subgroup $N$ of $\langle f^*_1, f^*_2 \rangle$ and an integral primitive element $v \in {\rm NS}\, (X) \cap \partial P({\rm NS}\, (X)) \setminus \{0\}$ such that $f^*(v) = v$ for all $f \in N$. As $N$ is a finite index subgroup of $\langle f^*_1, f^*_2 \rangle$ and $f^*_1, f^*_2 \in \langle f^*_1, f^*_2 \rangle$, there is a positive integer $n$ such that $(f^*_1)^n, (f^*_2)^n \in N$. Then $(f_i^*)^n(v_i) = v_i$ and $(f_i^*)^n(v) = v$. As $v_1$ and $v_2$ are $\Q$-linearly independent in ${\rm NS}\, (X)_{\Q}$, it follows that either $v_1$ and $v$ are $\Q$-linearly independent or $v_2$ and $v$ are $\Q$-linearly independent. By changing the order if necessary, we may assume that $v_1$ and $v$ are $\Q$-linearly independent. Then $(v_1 + v)^2 > 0$ by the Hodge index theorem. As $(f_i^*)^{n}(v_1 + v) = v_1 + v$, it follows that $(f_i^*)^n$ is of finite order on ${\rm NS}\, (X)$. This is because the orthogonal complement of $v_1 + v$ in ${\rm NS}\, (X)$ is negative definite by the Hodge index theorem and it is preserved by $(f_1^*)^n$. (Note that ${\rm O}\, (L, (*, **))$ is finite if the integral bilinear form $(*, **)$ is negative definite.) 

As $X$ is a {\it projective} K3 surface, it follows that $f_1$ is of finite order also as an element of ${\rm Aut}\, (X)$. This is because $f_1$ is then an automorphism of $X$ as a polarized manifold and $X$ has no global vector field. However, this contradicts to our assumption that $f_1$ is of infinite order in ${\rm MW}\, (\varphi_1) < {\rm Aut}\, (X)$. 
\end{proof}

The next proposition synthesizes Proposition \ref{prop21} (3) and the inertia group of a smooth rational curve that we are looking for:

\begin{proposition}\label{prop22} Under Set-up \ref{st21}, we assume further that there is a smooth rational curve $C \subset X$ such that $\varphi_i$ has a singular fiber $F_i$ in which $C$ is an irreducible component meeting at least three other irreducible components of $F_i$, i.e., $F_i$ is of Kodaira type $I_{b}^{*}$, $II^{*}$, $III^{*}$ or $IV^*$ for both $i=1$ and $2$. Then there is $f \in {\rm Ine}\, (C)$ such that $h_{{\rm top}}(f) > 0$.  (See \cite[Page 565]{Ko63} for the notation. For instance, if $F_1$ is of Kodaira type $IV^*$ and $F_2$ is of Kodaira type $III^*$, then $C$ is the irreducible component of multiplicity $3$ of $F_1$ and at the same time the irreducible component of multiplicity $4$ of $F_2$.)
\end{proposition}

\begin{proof} Let $C_{ij}$ ($1 \le j \le n_i$) be the set of irreducible components of $F_i$ ($i=1$, $2$). Then we have a natural group homomorphism 
$$\tau_i : {\rm MW}\, (\varphi_i) \to {\rm Aut}_{{\rm set}}(\{S_{ij}\}_{1 \le j \le n_i}) \simeq \Sigma_{n_i}\,\, .$$ 
Here $\Sigma_n$ is the symmetric group of $n$ letters. As $n_i$ is a natural number, ${\rm Ker}\, \tau_i$ is a finite index subgroup of ${\rm MW}(\varphi_i)$. Recall that ${\rm MW}\, (\varphi_i)$ ($i=1$, $2$) has an element of infinite order. Then there is $f_i \in {\rm Ker}\, \tau_i$ such that $f_i$ is of infinite order as well. By definition, we have $f_i(C_{ij}) = C_{ij}$ for all $j$. Thus $f_i(C) = C$ and $f_i(P_{i,k}) = P_{i, k}$ ($k = 1$, $2$, $3$). Here $P_{i,k} \in C$ ($k =1$, $2$, $3$) are the intersection points of $C$ with other (at least) three irreducible components of $F_i$. As $C \simeq \P^1$ and $P_{i, k}$ ($k =1$, $2$, $3$) are mutually distinct closed points on $C$, it follows that $f_i|_{C} = id_{C}$, i.e., $f_i \in {\rm Ine}\,(C)$. Hence 
$$\langle f_1, f_2 \rangle \subset {\rm Ine}\, (C)\,\, .$$
By Proposition \ref{prop21} (3), there is $f \in \langle f_1, f_2 \rangle$ such that $h_{{\rm top}}(f) > 0$. This $f$ satisfies all the requirements. 
\end{proof}
  
\section{Proof of Theorem \ref{thm1}}\label{sect3} 

In this section, we prove Theorem \ref{thm1}. 

Throughout this section, we denote by $S = {\rm Km}\, (E \times F)$ the Kummer surface associated to the product of two elliptic curves $E$ and $F$. 

Let $\{P_i\}_{i=1}^{4}$ (resp. $\{Q_i\}_{i=1}^{4}$) be the $2$-torsion subgroup of $F$ (resp. $E$). Then $S$ contains $24$ {\it visible} smooth rational curves. They are $8$ smooth rational curves $E_i$, $F_i$ ($1 \le i \le 4$) arizing from $8$ elliptic curves $E \times \{P_i\}$, $\{Q_i\} \times F$ on $E \times F$ and $16$ exceptional curves $C_{ij}$ over the $16$ singular points of type $A_1$ on the quotient surface $E \times F/\langle -id_{E \times F}\rangle$. We use the same notation as in \cite[Page 655]{Og89}. See Figure \ref{fig1} for the configuration of these $24$ visible smooth rational curves on $S$. 

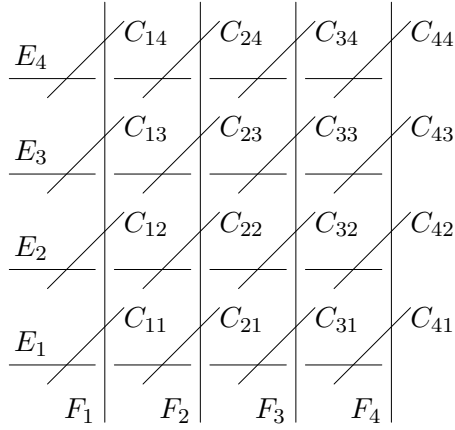
\begin{figure}
\unitlength 0.1in
\begin{picture}(25.000000,24.000000)(-1.000000,-23.500000)
\put(4.500000, -22.000000){\makebox(0,0)[rb]{$F_1$}}%
\put(9.500000, -22.000000){\makebox(0,0)[rb]{$F_2$}}%
\put(14.500000, -22.000000){\makebox(0,0)[rb]{$F_3$}}%
\put(19.500000, -22.000000){\makebox(0,0)[rb]{$F_4$}}%
\put(0.250000, -18.500000){\makebox(0,0)[lb]{$E_1$}}%
\put(0.250000, -13.500000){\makebox(0,0)[lb]{$E_2$}}%
\put(0.250000, -8.500000){\makebox(0,0)[lb]{$E_3$}}%
\put(0.250000, -3.500000){\makebox(0,0)[lb]{$E_4$}}%
\put(6.000000, -16.000000){\makebox(0,0)[lt]{$C_{11}$}}%
\put(6.000000, -11.000000){\makebox(0,0)[lt]{$C_{12}$}}%
\put(6.000000, -6.000000){\makebox(0,0)[lt]{$C_{13}$}}%
\put(6.000000, -1.000000){\makebox(0,0)[lt]{$C_{14}$}}%
\put(11.000000, -16.000000){\makebox(0,0)[lt]{$C_{21}$}}%
\put(11.000000, -11.000000){\makebox(0,0)[lt]{$C_{22}$}}%
\put(11.000000, -6.000000){\makebox(0,0)[lt]{$C_{23}$}}%
\put(11.000000, -1.000000){\makebox(0,0)[lt]{$C_{24}$}}%
\put(16.000000, -16.000000){\makebox(0,0)[lt]{$C_{31}$}}%
\put(16.000000, -11.000000){\makebox(0,0)[lt]{$C_{32}$}}%
\put(16.000000, -6.000000){\makebox(0,0)[lt]{$C_{33}$}}%
\put(16.000000, -1.000000){\makebox(0,0)[lt]{$C_{34}$}}%
\put(21.000000, -16.000000){\makebox(0,0)[lt]{$C_{41}$}}%
\put(21.000000, -11.000000){\makebox(0,0)[lt]{$C_{42}$}}%
\put(21.000000, -6.000000){\makebox(0,0)[lt]{$C_{43}$}}%
\put(21.000000, -1.000000){\makebox(0,0)[lt]{$C_{44}$}}%
\special{pa 500 2200}%
\special{pa 500 0}%
\special{fp}%
\special{pa 1000 2200}%
\special{pa 1000 0}%
\special{fp}%
\special{pa 1500 2200}%
\special{pa 1500 0}%
\special{fp}%
\special{pa 2000 2200}%
\special{pa 2000 0}%
\special{fp}%
\special{pa 0 1900}%
\special{pa 450 1900}%
\special{fp}%
\special{pa 550 1900}%
\special{pa 950 1900}%
\special{fp}%
\special{pa 1050 1900}%
\special{pa 1450 1900}%
\special{fp}%
\special{pa 1550 1900}%
\special{pa 1950 1900}%
\special{fp}%
\special{pa 0 1400}%
\special{pa 450 1400}%
\special{fp}%
\special{pa 550 1400}%
\special{pa 950 1400}%
\special{fp}%
\special{pa 1050 1400}%
\special{pa 1450 1400}%
\special{fp}%
\special{pa 1550 1400}%
\special{pa 1950 1400}%
\special{fp}%
\special{pa 0 900}%
\special{pa 450 900}%
\special{fp}%
\special{pa 550 900}%
\special{pa 950 900}%
\special{fp}%
\special{pa 1050 900}%
\special{pa 1450 900}%
\special{fp}%
\special{pa 1550 900}%
\special{pa 1950 900}%
\special{fp}%
\special{pa 0 400}%
\special{pa 450 400}%
\special{fp}%
\special{pa 550 400}%
\special{pa 950 400}%
\special{fp}%
\special{pa 1050 400}%
\special{pa 1450 400}%
\special{fp}%
\special{pa 1550 400}%
\special{pa 1950 400}%
\special{fp}%
\special{pa 200 2000}%
\special{pa 600 1600}%
\special{fp}%
\special{pa 200 1500}%
\special{pa 600 1100}%
\special{fp}%
\special{pa 200 1000}%
\special{pa 600 600}%
\special{fp}%
\special{pa 200 500}%
\special{pa 600 100}%
\special{fp}%
\special{pa 700 2000}%
\special{pa 1100 1600}%
\special{fp}%
\special{pa 700 1500}%
\special{pa 1100 1100}%
\special{fp}%
\special{pa 700 1000}%
\special{pa 1100 600}%
\special{fp}%
\special{pa 700 500}%
\special{pa 1100 100}%
\special{fp}%
\special{pa 1200 2000}%
\special{pa 1600 1600}%
\special{fp}%
\special{pa 1200 1500}%
\special{pa 1600 1100}%
\special{fp}%
\special{pa 1200 1000}%
\special{pa 1600 600}%
\special{fp}%
\special{pa 1200 500}%
\special{pa 1600 100}%
\special{fp}%
\special{pa 1700 2000}%
\special{pa 2100 1600}%
\special{fp}%
\special{pa 1700 1500}%
\special{pa 2100 1100}%
\special{fp}%
\special{pa 1700 1000}%
\special{pa 2100 600}%
\special{fp}%
\special{pa 1700 500}%
\special{pa 2100 100}%
\special{fp}%
\end{picture}%
 \caption{Curves $E_i$, $F_j$ and $C_{ij}$}
 \label{fig1}
\end{figure}

Thoughout this section, we set:
$$C := E_4 \subset S\,\, .$$
\begin{proposition}\label{prop31} $[{\rm Dec}\, (C) : {\rm Ine}\, (C)] = \infty$ and there is $f \in {\rm Ine}\, (C)$ such that $h_{{\rm top}}(f) > 0$. 
\end{proposition}

\begin{proof}
Consider the following three divisors of Kodaira type $IV^*$ on $S$: 
$$D_1 := F_1 + 2C_{14} + F_2 + 2C_{24} + F_3 + 2C_{34} + 3E_4\,\, ,$$
$$D_2 := F_1 + 2C_{14} + F_2 + 2C_{24} + F_4 + 2C_{44} + 3E_4\,\, ,$$
$$D_3 := E_4 + 2C_{44} + E_3 + 2C_{43} + E_2 + 2C_{42} + 3F_4\,\, .$$
Then the linear system $|D_k|$ is a free pencil and defines an elliptic fibration 
$$\varphi_k : S \to \P^1$$
with a singular fiber $D_k$ of Kodaira type $IV^*$, for each $k =1$, $2$, $3$. 
Notice that the three fibrations $\varphi_k$ are mutually different and each $\varphi_k$ admits a section as we shall see. 

First, we show that $[{\rm Dec}\, (E_{4}) : {\rm Ine}\, (E_{4})] = \infty$. Recall that we set $C = E_{4}$. 

The curves $C_{14}$ and $C_{24}$ are disjoint sections of $\varphi_3$ meeting the same irreducible component $E_4$ of the fiber $D_3$ of $\varphi_3$. We choose $C_{14}$ as the zero section of $\varphi_3$. Then the section $C_{24}$ defines 
$f_3 \in {\rm MW}(\varphi_1)$, which preserves each irreducible component of $D_3$. In particular, $f_3(E_4) = E_4$. Note that 
$$E_4^{0} := E_4 \setminus C_{44} = \C\,\, .$$ 
Then, by \cite[Table 1, Page 604]{Ko63}, if we set
$$0 := C_{14} \cap E_4^{0} \subset E_4^{0} = \C\,\, ,$$
$$a := C_{24} \cap E_4^{0} \subset E_4^{0} = \C\,\, ,$$
then $f_3(E_4) = E_4$ and $f_3$ acts on $E_4^0 = \C$ by the addition $x \mapsto x+a$. As $C_{14} \cap C_{24} = \emptyset$, it follows that $a \not= 0$, hence $f_3$ is of infinite order on $E_4$. As $f_3 \in {\rm Dec}\, (E_{4})$ and $f_3|_{E_4}$ is of infinite order on $E_4$, it follows that the class $f_3$ is of infinite order in the quotient group ${\rm Dec}\, (E_{4})/{\rm Ine}\, (E_{4})$. Hence $[{\rm Dec}\, (E_{4}) : {\rm Ine}\, (E_{4})] = \infty$.

Next we show that there is $f \in {\rm Ine}\, (C)$ such that $h_{{\rm top}}(f) > 0$. 

Note that $C = E_4$ is the irreducible component of both $D_1$ and $D_2$ meeting three other irreducible components of both $D_1$ and $D_2$. Note also that $C_{11}$ and $C_{12}$ are disjoint section of both $\varphi_k$ ($k=1$, $2$), meeting the same irreducible component $F_1$ of the fiber $D_k$. Thus, for the same reason above, $C_{11}$ and $C_{12}$ define the element $f_k \in {\rm MW}\, (\varphi_k)$ of infinite order for each $k =1$ and $2$. As $C = E_{4}$ is the irreducible component of both $D_k$ ($k=1$, $2$) meeting three other irreducible components of $D_k$, there is then $f \in {\rm Ine}\, (E_4)$ such that $h_{{\rm top}}(f) > 0$ by Proposition \ref{prop22}. 
\end{proof}

The following proposition completes the proof of Theorem \ref{thm1}: 
\begin{proposition}\label{prop32} Let $S = {\rm Km}\, (E \times F)$ and $\P^1 \simeq C = E_4 \subset S$ as before. Assume that the elliptic curves $E$ and $F$ are not isogenous. Then $[{\rm Aut}\, (S) : {\rm Dec}\, (C)] < \infty$. 
\end{proposition}

\begin{proof} Consider the element $\iota \in {\rm Aut}\, (S)$ of order $2$, induced by the element $(-id_E, id_F) \in {\rm Aut}\, (E \times F)$. As $E$ and $F$ are not isogenous, it follows from \cite[Lemma 1.4]{Og89} that $\iota$ is in the center of ${\rm Aut}\, (S)$ and the pointwisely fixed locus of $\iota$ is 
$$S^{\iota} = \cup_{i=1}^{4} E_i \cup \cup_{i=1}^{4} F_i\,\, .$$ 
Thus $S^{\iota}$ is preserved by ${\rm Aut}\, (S)$ and we obtain a natural group homomorphism 
$$\tau : {\rm Aut}\, (S) \to {\rm Aut}_{{\rm set}}(\{E_i, F_i\}_{i=1}^{4}) \simeq \Sigma_8\,\, .$$
Here, as before $\Sigma_n$ is the symmetric group of $n$ letters. As $\Sigma_8$ is a finite group, it follows that $[{\rm Aut}\, (S) : {\rm Ker}\, \tau] < \infty$. By definition of ${\rm Dec}\, (E_4)$, we have 
$${\rm Ker}\, \tau < {\rm Dec}\, (E_4) < {\rm Aut}\, (S)\,\, .$$ 
Thus $[{\rm Aut}\, (S) : {\rm Dec}\, (E_4)] < \infty$ as well. 
\end{proof}
\begin{remark}\label{rem32}
Let $S := {\rm Km}\, (E \times E)$. Then $[{\rm Aut}\, (S) : {\rm Dec}\, (C)] = \infty$ ($C = E_4$). This is observed as follows. 
Let $\Delta \in {\rm Aut}\, (E \times E)$ defined by $(x, y) \mapsto (x, y+x)$. Consider the automorphism $\delta \in {\rm Aut}\, (S)$ induced by $\Delta$. Then $\delta$ is of infinite order but $\delta^n(E_4) = E_4$ only if $n = 0$. 
\end{remark}  
\begin{remark}\label{rem31} 
A projective K3 surface $V$ is called {\it $2$-elementary} if $V$ has an automorphism $\iota$ of order $2$ such that $\iota^*\omega_V  = -\omega_V$, or equivalently $\iota^*|_{T(S)} = -id_{T(V)}$, and $\iota^* = id$ on ${\rm NS}\, (V)$. The name comes from the fact that the discriminant group ${\rm NS}\, (V)^*/{\rm NS}\, (V)$ is then isomorphic to some $2$-elemenray group $(\Z/2)^{\oplus a(V)}$. $2$-elemenray K3 surfaces are intensively studied by Nikulin (\cite{Ni81}). The involution $\iota$ is in the center of ${\rm Aut}\, (V)$ and the fixed locus of $\iota$ is preserved under ${\rm Aut}\, (V)$. By the classification in \cite{Ni81}, if a $2$-elementary K3 surface has a smooth rational curve $C$ and an automorphism of positive entropy, then $V^{\iota}$ has to be either empty or consists of smooth rational curves (which are necessarily disjoint). The second case occurs exactly when $\rho(V) + a(V) = 22$ by \cite{Ni81}. 
Kummer surfaces ${\rm Km}\, (E \times F)$ of non-isogenous $E$ and $F$ and K3 surfaces of classical Coble type are two special cases of $2$-elementary K3 surfaces with $(\rho(V), a(V)) = (18, 4)$, $(11, 11)$ respectively. It may be interesting to check if any $2$-elementary K3 surface with $\rho(V) + a(V) = 22$ has a smooth rational curve $C$ whose inertia group has an automorphism of positive entropy or not. See also Remark \ref{rem41}. 
\end{remark}

\section{Proof of Theorems \ref{thm2} and \ref{cor1}}\label{sect4} 

In this section, we show Theorems \ref{thm2} and \ref{cor1}. 

Throughout this section, $\zeta_3 := (-1 + \sqrt{-3})/2$ is a primitive 3rd root of unity, $E = E_{3}$ is the ellitic curve of period $\zeta_3$ and $S = S_3$ is the minimal resolution of the quotient surface $(E \times E)/\langle (\zeta_3, \zeta_3^2) \rangle$. Then $S$ is a projective K3 surface with $24$ {\it visible} smooth rational curves $F_i$, $G_i$ ($i = 1$, $2$, $3$), $E_{ij}$, $E_{ij}'$ ($i =1$, $2$, $3$, $j=1$, $2$, $3$) whose configuration is in Figure \ref{fig2}. We follow \cite[Page 1281]{OZ96} for the notation of these $24$ curves. 

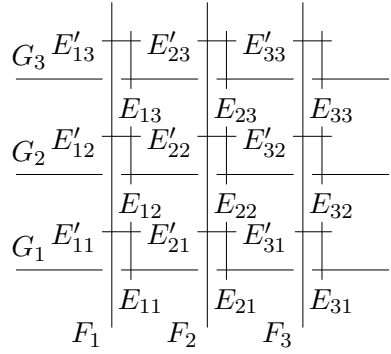
\begin{figure}
\unitlength 0.1in
\begin{picture}(25.000000,24.000000)(-1.000000,-23.500000)
\put(4.500000, -23.000000){\makebox(0,0)[rb]{$F_1$}}%
\put(9.500000, -23.000000){\makebox(0,0)[rb]{$F_2$}}%
\put(14.500000, -23.000000){\makebox(0,0)[rb]{$F_3$}}%

\put(-0.250000, -18.500000){\makebox(0,0)[lb]{$G_1$}}%
\put(-0.250000, -13.500000){\makebox(0,0)[lb]{$G_2$}}%
\put(-0.250000, -8.500000){\makebox(0,0)[lb]{$G_3$}}%

\put(3.000000, -16.500000){\makebox(0,0)[bt]{$E_{11}'$}}%
\put(3.000000, -11.500000){\makebox(0,0)[bt]{$E_{12}'$}}%
\put(3.000000, -6.500000){\makebox(0,0)[bt]{$E_{13}'$}}%

\put(8.000000, -16.500000){\makebox(0,0)[bt]{$E_{21}'$}}%
\put(8.000000, -11.500000){\makebox(0,0)[bt]{$E_{22}'$}}%
\put(8.000000, -6.500000){\makebox(0,0)[bt]{$E_{23}'$}}%

\put(13.000000, -16.500000){\makebox(0,0)[bt]{$E_{31}'$}}%
\put(13.000000, -11.500000){\makebox(0,0)[bt]{$E_{32}'$}}%
\put(13.000000, -6.500000){\makebox(0,0)[bt]{$E_{33}'$}}%

\put(6.500000, -20.000000){\makebox(0,0)[bt]{$E_{11}$}}%
\put(6.500000, -15.000000){\makebox(0,0)[bt]{$E_{12}$}}%
\put(6.500000, -10.000000){\makebox(0,0)[bt]{$E_{13}$}}%

\put(11.500000, -20.000000){\makebox(0,0)[bt]{$E_{21}$}}%
\put(11.500000, -15.000000){\makebox(0,0)[bt]{$E_{22}$}}%
\put(11.500000, -10.000000){\makebox(0,0)[bt]{$E_{23}$}}%

\put(16.500000, -20.000000){\makebox(0,0)[bt]{$E_{31}$}}%
\put(16.500000, -15.000000){\makebox(0,0)[bt]{$E_{32}$}}%
\put(16.500000, -10.000000){\makebox(0,0)[bt]{$E_{33}$}}%

\special{pa 500 2200}%
\special{pa 500 500}%
\special{fp}%
\special{pa 1000 2200}%
\special{pa 1000 500}%
\special{fp}%
\special{pa 1500 2200}%
\special{pa 1500 500}%
\special{fp}%

\special{pa 0 1900}%
\special{pa 450 1900}%
\special{fp}%
\special{pa 550 1900}%
\special{pa 950 1900}%
\special{fp}%
\special{pa 1050 1900}%
\special{pa 1450 1900}%
\special{fp}%
\special{pa 1550 1900}%
\special{pa 1950 1900}%
\special{fp}%
\special{pa 0 1400}%
\special{pa 450 1400}%
\special{fp}%
\special{pa 550 1400}%
\special{pa 950 1400}%
\special{fp}%
\special{pa 1050 1400}%
\special{pa 1450 1400}%
\special{fp}%
\special{pa 1550 1400}%
\special{pa 1950 1400}%
\special{fp}%
\special{pa 0 900}%
\special{pa 450 900}%
\special{fp}%
\special{pa 550 900}%
\special{pa 950 900}%
\special{fp}%
\special{pa 1050 900}%
\special{pa 1450 900}%
\special{fp}%
\special{pa 1550 900}%
\special{pa 1950 900}%
\special{fp}%

\special{pa 450 700}%
\special{pa 650 700}%
\special{fp}%
\special{pa 950 700}%
\special{pa 1150 700}%
\special{fp}%
\special{pa 1450 700}%
\special{pa 1650 700}%
\special{fp}%

\special{pa 450 1700}%
\special{pa 650 1700}%
\special{fp}%
\special{pa 950 1700}%
\special{pa 1150 1700}%
\special{fp}%
\special{pa 1460 1700}%
\special{pa 1650 1700}%
\special{fp}%

\special{pa 450 1200}%
\special{pa 650 1200}%
\special{fp}%
\special{pa 950 1200}%
\special{pa 1150 1200}%
\special{fp}%
\special{pa 1450 1200}%
\special{pa 1650 1200}%
\special{fp}%

\special{pa 600 1650}%
\special{pa 600 1950}%
\special{fp}%
\special{pa 600 1150}%
\special{pa 600 1450}%
\special{fp}%
\special{pa 600 650}%
\special{pa 600 950}%
\special{fp}%

\special{pa 1100 1650}%
\special{pa 1100 1950}%
\special{fp}%
\special{pa 1100 1150}%
\special{pa 1100 1450}%
\special{fp}%
\special{pa 1100 650}%
\special{pa 1100 950}%
\special{fp}%

\special{pa 1600 1650}%
\special{pa 1600 1950}%
\special{fp}%
\special{pa 1600 1150}%
\special{pa 1600 1450}%
\special{fp}%
\special{pa 1600 650}%
\special{pa 1600 950}%
\special{fp}%

\end{picture}%
 \caption{Curves $F_i$, $G_j$, $E_{ij}$ and $E_{ij}'$}
 \label{fig2}
\end{figure}

Throughout this section, we also denote by $g = g_3$ the automorphism of $S$ of order $3$, induced by the automorphism $\tilde{g} \in {\rm Aut}\, (E \times E)$ defined by 
$$\tilde{g}^{*}(x, y) = (x, \zeta_3 y)\,\, .$$ 
Here $(x, y)$ are the standard coordinates of the universal covering space $\C \times \C$ of $E \times E$. Then by an explicit computation or by using \cite[Example 1, Theorem 3, Lemma 2.3]{OZ96}, we have the following:
\begin{proposition}\label{prop41} The pair $(S,g)$ satisfies the following:

\begin{enumerate} 
\item Each of the $24$ visible smooth rational curves above, say $D$, is preserved by $g$, i.e., $g \in {\rm Dec}\, (D)$. Moreover the fixed locus $X^g$ consists of $6$ smooth rational curves $F_i$, $G_i$ and $9$ isolated points $P_{ij} = E_{ij} \cap E_{ij}'$ ($1 \le i, j \le 3$). 
\item $\rho(S) = 20$, $g^* = id$ on ${\rm NS}\, (S)$ and $g^*\omega_S = \zeta_3 \omega_S$. 
\end{enumerate}
\end{proposition}

Set ${\rm Aut}^0\, (S) := \{f \in {\rm Aut}\, (S)\, |\, f^*\omega_S = \omega_S\}$. As $S$ is projective, the pluricanonical representation of ${\rm Aut}\, (S)$ is of finite order (\cite[Theorem 14.10]{Ue75}). Thus ${\rm Aut}^0\, (S)$ is a normal subgroup of ${\rm Aut}\, (S)$ of finite index. 

From now until the end of this section, we set
$$C = G_3 \simeq \P^1\,\, .$$
We show that $C$ satisfies the requirements in Theorem \ref{thm2}. 
 
\begin{proposition}\label{prop42}
$[{\rm Aut}\,(S) : {\rm Dec}\, (C)] < \infty$. 
\end{proposition}

\begin{proof}
As ${\rm Aut}^0\, (S)$ is a normal subgroup of ${\rm Aut}\, (S)$ of finite index, it suffices to show that 
$$[{\rm Aut}^0\, (S) : {\rm Aut}^0\, (S) \cap {\rm Dec}\, (C)] < \infty\,\, .$$
Indeed, for a group $G$ and a normal subgroup $N <G$ and a subgroup $H < G$, we have an obvious inequality:
$$[G : W] \le [N : N \cap W] \cdot [G : N]\,\, .$$
As $g^* = id$ on ${\rm NS}\, (S)$ by Proposition \ref{prop41} (2), and $f^* = id$ on $T(S)$ for $f \in {\rm Aut}^0\, (S)$, it follows that $(f\circ g)^* = (g \circ f)^*$ on ${\rm NS}\, (S) \oplus T(S)$ for $f \in {\rm Aut}^0\, (S)$. As ${\rm NS}\, (S) \oplus T(S)$ is a finite index subgroup of $H^2(S, \Z)$, it follows that $(f \circ g)^* = (g \circ f)^*$ on $H^2(S, \Z)$. Hence $f \circ g = g \circ f$ by the global Torelli theorem for K3 surfaces. Then $f(S^g) = S^g$ for $f \in {\rm Aut}^0\, (S)$. Therefore, by Proposition \ref{prop41} (1), we have a natural group homomorphism 
$$\tau : {\rm Aut}^0\, (S) \to {\rm Aut}_{{\rm set}}(\{F_i, G_i\}_{i=1}^{3}) \simeq \Sigma_6\,\, .$$
As ${\rm Ker}\, \tau$ is of finite index in ${\rm Aut}^0\, (S)$ and 
$${\rm Ker}\, \tau < {\rm Dec}\, (G_3) \cap {\rm Aut}^0\, (S) < {\rm Aut}^0\, (S)\,\, ,$$
the result follows. 
\end{proof}

\begin{proposition}\label{prop43}
$[{\rm Dec}\, (C) : {\rm Ine}\, (C)] = \infty$. More strongly, there is $h \in {\rm Dec}\, (C)$ such that $h|_{C}$ is of infinite order and $h \in {\rm Aut}^{0}\, (S)$. 
\end{proposition}
\begin{proof}
Consider the elliptic fibration $\varphi : S \to \P^1$ 
given by the divisor $D$ of Kodaira type $III^*$:
$$D := G_3 + 2E_{33} + 3E_{33}' + 4F_3 + 3E_{31}' + 2E_{31} + G_1 + 2E_{32}'\,\, .$$
Then $E_{13}$ and $E_{23}$ are disjoint sections of $\varphi$ meeting the singular fiber $D$ of $\varphi$ of Kodaira type $III^*$ at the same irreducible component $G_3$. 
We choose $E_{13}$ as the zero section of $\varphi$. Let $h \in {\rm MW}\, (\varphi)$ be the element defined by $E_{23}$. Then $h \in {\rm Dec}\, (C)$ and $h|_C$ is of infinite order for the same reason as in the proof of Proposition \ref{prop31}. Hence $[{\rm Dec}\, (C) : {\rm Ine}\, (C)] = \infty$, again for the same reason as in the proof of Proposition \ref{prop31}. We have also $h^*\omega_S = \omega_S$ as $h \in {\rm MW}\, (\varphi)$. 
\end{proof}

\begin{proposition}\label{prop44}
There is $f \in {\rm Ine}\, (C)$ such that $h_{{\rm top}}(f) >0$ and $f \in {\rm Aut}^{0}\, (S)$. 
\end{proposition}

\begin{proof}
Consider the elliptic fibrations $\varphi_k : S \to \P^1$ ($k=1$, $2$) given respectively by the divisors of Kodaira type $III^*$
$$D_1 := F_1 + 2E_{13}' + 3E_{13} + 4G_3 + 3E_{23} + 2E_{23}' + F_2 + 2E_{33}\,\, ,$$
$$D_2 := F_1 + 2E_{13}' + 3E_{13} + 4G_3 + 3E_{33} + 2E_{33}' + F_3 + 2E_{23}\,\, .$$
Then $E_{11}'$ and $E_{12}'$ are disjoint sections of both $\varphi_k$ meeting the irreducible component $F_1$ of the singular fiber $D_k$ of $\varphi_k$. 
We choose $E_{11}'$ as the zero section of both $\varphi_k$ and consider the element $f_k \in {\rm MW}\,(\varphi_k)$ defined by the section $E_{12}'$. 
Then for the same reason as in the proof of Proposition \ref{prop31}, $f_k$ is of infinite order in ${\rm MW}\,(\varphi_k)$ and is also an element of ${\rm Ine}\, (G_3)$. Note that $C = G_3$ is the irreducible component of $D_k$ meeting three other irreducible components of $D_k$ for both $k=1$ and $2$. Then there is $f \in \langle f_1, f_2 \rangle$ with $h_{{\rm top}}(f) >0$ by Proposition \ref{prop22}. Note also that $f^*\omega_S = \omega_S$. This is because $f_k^*\omega_S = \omega_S$ as $f_k \in {\rm MW}\, (\varphi_k)$ ($k =1$, $2$). This completes the proof. 
\end{proof}

Now Theorem \ref{thm2} has been proved. 

\begin{remark}\label{rem41} As ${\rm NS}(S_3)^*/{\rm NS}\,(S_3) \simeq \Z/3$ (\cite{Vi83}), the most algebraic K3 surface $S_3$ is not a $2$-elementary K3 surface in the sense of \cite{Ni81}. As it is pointed by Xun Yu, by using a similar method, one can show that the other most algebraic K3 surface $S_4$ (\cite{Vi83}) also satisfies the same property as in Theorem \ref{thm2}. We leave details to the readers. We note that $S_4$ is a $2$-elementary K3 surface with $(\rho(S_4), a(S_4)) = (20, 2)$ (see Remark \ref{rem31} for $a(S_4)$) under a non-symplectic involution $g_2$ and the quotient surface  $S_4/\langle g_2 \rangle$ is a smooth non-classical Coble surface. 
\end{remark}

Next we show Theorem \ref{cor1}. We continue to use the same notations introduced at the beginning of this section. 

Let $\nu : W \to \overline{W}$ be the minimal resolution of the quotient surface $\overline{W} := S/\langle g \rangle$ and $\pi : S \to \overline{W}$ the quotient map. By the description of $S^{g}$, the surface $\overline{W}$ has exactly $9$ singular points of type $\frac{1}{3}(1,1)$ at $\pi(P_{jk})$. Then $C_{jk} := \nu^{-1}(\pi(P_{jk}))_{{\rm red}}$ is a smooth rational curve with $(C_{jk}, C_{jk}) = -3$. We denote the smooth locus of $\overline{W}$ by $\overline{W}^0$, i.e., $\overline{W}^0 := \overline{W} \setminus \{\pi(P_{jk})\}_{j, k = 1}^{3}$. We set $\overline{F}_i := \pi(F_i)_{{\rm red}}$ and $\overline{G}_i := \pi(G_i)_{{\rm red}}$ ($1 \le i \le 3$). These $6$ curves are mutulally disjoint smooth rational curves in $\overline{W}^{0}$ with self-intersection number $-6$. As $\overline{F}_i, \overline{G}_i \subset \overline{W}^0$, the morphism $\nu$ is isomorphic around $\overline{F}_i$ and $\overline{G}_i$. For this reason, we denote $\nu^{*}(\overline{F}_i)$ and $\nu^{*}(\overline{G}_i)$ by the same letter $\overline{F}_i$ and $\overline{G}_i$. 

\begin{proposition}\label{prop45}
The surface $W$ is a smooth rational surface such that $|mK_W| = \emptyset$ for $m = -1$ and $-2$ and $|-3K_W|$ consists of a unique element. More precisely, 
$$|-3K_W| = \{2\sum_{i=1}^{3} (\overline{F}_i + \overline{G}_i) + \sum_{j, k = 1}^{3} C_{jk}\}\,\, .$$ 
\end{proposition}

\begin{proof} As $q(S) = 0$ and $\nu^{-1} \circ \pi : S \dasharrow W$ is a dominant rational map between smooth varieties, it follows that $q(W) = 0$. Here $q(W)$ and $q(S)$ are the irregularities of $W$ and $S$. 

As $\pi$ is a finite cyclic cover of degree $3$, being totally ramified along $\sum_{i=1}^{3} (F_{i} + G_{i})$ over $\overline{W}^{0}$ and $K_S$ is trivial, it follows from the ramification formula that $-K_{\overline{W}}$ is $\Q$-linearly equivalent to an effective $\Q$-divisor: 
$$\frac{2}{3}\sum_{i=1}^{3} (\overline{F}_i + \overline{G}_i)\,\, \sim_{\Q}\,\, -K_{\overline{W}}\,\, .$$
In particular, $mK_{\overline{W}} \not\sim_{\Q} 0$ for any integer 
$m \not= 0$. Here and hereafter, we denote by $\sim_{\Q}$ the $\Q$-linear equivalence of Weil divisors. 

Assume that $|mK_W| \not= \emptyset$ for some integer 
$m \not= 0$. 

Then there is $D \in |mK_W|$. Let $\overline{D} := \nu_*D$. Then $\overline{D}$ is a {\it non-zero} effective Weil divisor on $\overline{W}$ and it is linearly equivalent to $mK_{\overline{W}}$, as $\overline{W}$ is normal and $\nu$ is isomorphic over $\overline{W}^{0}$. We also note that as $\overline{W}$ has only quotient singularities, any Weil divisor on $\overline{W}$ is $\Q$-Cartier.  

As $\overline{W}$ is normal and projective and $-K_{\overline{W}}$ is $\Q$-linearly equivalent to a non-zero effective divisor, it follows that $m < 0$. In particular, $|mK_W| = \emptyset$ for all positive integer $m >0$. As in addition $q(W) = 0$, $W$ is a rational surface by Castelnuovo's criterion. 

From now, we assume that $m <0$. Set $n = -m$. Then $n$ is a positive integer. 
Assume that $M \in |-nK_{\overline{W}}|$. 
Note that $-nK_{\overline{W}}$ is $\Q$-linearly equivalent to an effective $\Q$-divisor
$$\frac{2n}{3} \sum_{i=1}^{3} (\overline{F}_i + \overline{G}_i)\,\, .$$
As the $6$ curves $\overline{F}_i$ and $\overline{G}_i$ are mutually disjoint curves with $(\overline{F}_i, \overline{F}_i) = (\overline{G}_i, \overline{G}_i) = -6 <0$, it follows that $(M, \overline{F}_i) <0$ and $(M, \overline{G}_i) < 0$ for $i =1$, $2$, $3$. Then 
$$\overline{F}_i, \overline{G}_i \subset {\rm Supp}\, M\,\, ,$$ 
as $M$ is an effective Weil divisor. Hence 
$$M_1 := M - \sum_{i=1}^{3}(\overline{F}_i + \overline{G}_i)\,\, (\,\, \sim_{\Q}\,\, \frac{2n-3}{3} \sum_{i=1}^{3}(\overline{F}_i + \overline{G}_i)\,\, )$$
 is an effective Weil divisor. Thus $n \ge 2$. If $n \ge 2$, then $(M_1, \overline{F}_i) < 0$ and $(M_1, \overline{G}_i) < 0$. Therefore, 
for the same reason above, 
$$M_2 := M_1 - \sum_{i=1}^{3}(\overline{F}_i + \overline{G}_i) = M - 2\sum_{i=1}^{3}(\overline{F}_i + \overline{G}_i)\,\, (\,\, \sim_{\Q}\,\, \frac{2n-6}{3} \sum_{i=1}^{3}(\overline{F}_i + \overline{G}_i)\,\, )$$ 
is also an effective divisor. Thus $n \ge 3$ and if $n=3$, then $M_2 = 0$ and therefore $|-3K_{\overline{W}}|$ consists of the single element $2\sum_{i=1}^{3}(\overline{F}_i + \overline{G}_i)$. 

Hence, by the observation made at the biginning, $|-K_{W}| = \emptyset$ and $|-2K_{W}| = \emptyset$, and moreover if $D \in |-3K_{W}|$ then 
$$\nu_{*}D = 2\sum_{i=1}^{3}(\overline{F}_i + \overline{G}_i)$$
as a divisor. Hence we have
$$D = 2\sum_{i=1}^{3}(\overline{F}_i + \overline{G}_i) + \sum_{j, k = 1}^{3} a_{jk}C_{jk}$$
for some non-negative integers $a_{jk}$ as a divisor. Here, the integers $a_{jk}$ are uniquely determined, as $C_{jk}$ are mutually disjoint curves with $(C_{jk}, C_{jk}) = -3$ and $(C_{jk}, \overline{F}_i) = (C_{jk}, \overline{G}_i) = 0$. 
Hence the element $|-3K_{W}|$ is unique if exists. On the other hand, as $\nu$ is a minimal resolution of $9$ singular points of type $1/3(1,1)$, we know by the adjunction formula that $-3K_{W}$ is linearly equivalent to the effective divisor
$$2\sum_{i=1}^{3}(\overline{F}_i + \overline{G}_i) + \sum_{j, k = 1}^{3} C_{jk}\,\, .$$
This completes the proof. 
\end{proof}

Let $\overline{C} := \overline{G}_3 \subset W$. Then $\overline{C}$ is a smooth rational curve on $W$. 
The next proposition completes the proof of Theorem \ref{cor1}. 
\begin{proposition}\label{prop46} The curve $\overline{C}$ satisfies

\begin{enumerate}
\item 
$[{\rm Aut}\, (W) : {\rm Dec}\, (\overline{C})] < \infty$ and $[{\rm Dec}\, (\overline{C}) : {\rm Ine}\, (\overline{C})] = \infty$.
\item There is $f_W \in {\rm Ine}\, (\overline{C})$ such that $h_{{\rm top}}(f_W) > 0$. 
\end{enumerate}

\end{proposition}

\begin{proof} As ${\rm Aut}\, (W)$ acts on $|-3K_{W}|$, the divisor  
$$2\sum_{i=1}^{3}(\overline{F}_i + \overline{G}_i) + \sum_{j, k = 1}^{3} C_{jk}$$ 
is stable under ${\rm Aut}\, (W)$ by Proposition \ref{prop45}. Therefore we have a natural representaion 
$$\tau : {\rm Aut}\, (W) \to {\rm Aut}_{{\rm set}}(\{\overline{F}_i, \overline{G}_i\}_{i=1}^{3}) \simeq \Sigma_6\,\, .$$
Then $[{\rm Aut}\, (W) : {\rm Ker}\, \tau] < \infty$. As ${\rm Ker}\, \tau < {\rm Dec}\, (\overline{G}_3) < {\rm Aut}\, (W)$, we have $[{\rm Aut}\, (W) : {\rm Dec}\, (\overline{G}_3)] < \infty$.   

We show that $[{\rm Dec}\, (\overline{C}) : {\rm Ine}\, (\overline{C})] = \infty$. Consider the element $h \in {\rm Aut}\, (S)$ in Proposition \ref{prop43}. Recall that $g^* = id$ on ${\rm NS}\, (S)$ and $h^* = id$ on $T(S)$ as $h \in {\rm Aut}^{0}\, (S)$. Thus $h \circ g = g \circ h$ for the same reason as in the proof of Proposition \ref{prop42}. It follows that $h$ induces an automorphism of $h_W \in {\rm Aut}\, (W)$. By definition of $h_W$, the map $\nu^{-1} \circ \pi : S \dasharrow W$ is equivariant under the actions $h$ and $h_W$. As $\overline{C} = \nu^{-1} \circ \pi (G_3)$, $h(G_3) = G_3$ and $h|_{G_3}$ is of infinite order in ${\rm Aut}\, (G_3)$, it follows that $h_W(\overline{C}) = \overline{C}$ and $h_W|_{\overline{C}}$ is of infinite order in ${\rm Aut}\, (\overline{C})$. Hence the class of $h_{W}$ is of infinite order in the quotient group ${\rm Dec}\, (\overline{C})/{\rm Ine}\, (\overline{C})$. This shows $[{\rm Dec}\, (\overline{C}) : {\rm Ine}\, (\overline{C})] = \infty$. 

We show the assertion (2). Let $f \in {\rm Aut}\, (S)$ be an automorphism found in Proposition \ref{prop44}. As $g^* = id$ on ${\rm NS}\, (S)$ and $f \in {\rm Aut}^{0}\, (S)$, for the same reason as above, $f$ induces an automorphism $f_W \in {\rm Aut}\, (W)$ of $W$. As $f \in {\rm Ine}\, (G_3)$, i.e., $f|_{G_3} = id_{G_3}$, we have $f_W|_{\overline{C}} = id_{\overline{C}}$, i.e., $f_W \in {\rm Ine}\, (\overline{C})$, again for the same reason as above. Moreover, as $\nu^{-1} \circ \pi : S \dasharrow W$ is a generically finite dominant rational map which is equivariant under the actions of $f$ and $f_{W}$, 
we have 
$$h_{{\rm top}}(f_W) = h_{{\rm top}}(f) > 0$$ 
by \cite[Corollary 1.2]{DN11}. This completes the proof.  
\end{proof}  

\section{Proof of Theorem \ref{thm3}}\label{sect5} 

In this section, we show Theorem \ref{thm3}. 

In \cite{Og10}, we find the following:

\begin{theorem}\label{thm51} There is a (necessarily non-projective) K3 surface $S$ with an automorphism $g$ satisfying the following properties:

\begin{enumerate}
\item ${\rm Aut}\, (S) = \langle g \rangle \simeq \Z$. 
\item The topological entropy of $g$ is the natural logarithm of the Salem number 
$$\alpha = 1.200026\ldots\,\, ,$$ 
which is the unique real root $> 1$ of the following irreducible monic polynomial in the polynomial ring $\Z [x]$:
$$\varphi_{14}(x) = x^{14} -x^{11} - x^{10} + 
x^{7} - x^{4} -x^{3} + 1\,\, .$$  
\item Complete irreducible curves on $S$ are exactly $8$ smooth rational curves $C_i$ ($1 \le i \le 8$) and 
$${\rm NS}\, (S) = \Z \langle C_i \rangle_{i=1}^{8} \simeq E_8\,\, ,$$
i.e., the intersection number $(C_i, C_j)$ ($1 \le i < j \le 8$) is $0$ or $1$ and it is $1$ exactly when $(i, j)$ is in 
$$\{(1, 2)\,\, ,\,\, (2, 3)\,\, ,\,\,(3, 4)\,\, ,\,\, (4, 5)\,\, ,\,\,(5, 6)\,\, ,\,\, (6, 7)\,\, ,\,\,(3, 8)\}\,\, .$$
\item The fixed locus $S^g$ of $g$ consists of $C := C_3$ and $7$ isolated points $Q_i$ ($1 \le i \le 7$) on $\cup_{i=1}^{8} C_ i \setminus C$ and one point $Q$ outside $\cup_{i=1}^{8}C_i$. 
\end{enumerate}
\end{theorem}

We show that the K3 surface $S$ in Theorem \ref{thm51} satisfies the conditions (1) and (2) in Theorem \ref{thm3}. 

As the N\'eron-Severi lattice $({\rm NS}\, (S), (*, **))$ is negative definite by Theorem \ref{thm51} (3), it follows that $S$ has no non-constant global meromorphic functions. In particular, $S$ is non-projective. 

Consider the curve $C := C_3 \simeq \P^1$. We show that $C$ satisfies the condition (1) in Theorem \ref{thm3}. 

As the Dynkin diagram $E_8$ has no non-trivial automorphism, ${\rm Aut}\, (S)$ preserves each curve $C_i$ ($1 \le i \le 8$). As $C = C_3$ meets $C_2$, $C_4$, $C_8$ and $C_3 \cap C_4$, $C_3 \cap C_4$, $C_3 \cap C_8$ are three mutually different points on $C_3 \simeq \P^1$, it follows that ${\rm Aut}\, (S) = {\rm Ine}\, (C)$. Then $g \in {\rm Ine}\, (C)$ and $h_{{\rm top}}(g) > 0$ by Theorem \ref{thm51} (2). 

Consider the curve $D := C_8 \simeq \P^1$. We show that $D$ satisfies the condition (2) in Theorem \ref{thm3}. 

As ${\rm Aut}\, (S) = \langle g \rangle$ and $g$ preserves $D$, it follows that ${\rm Aut}\, (S) = {\rm Dec}\, (D)$. 

Now consider ${\rm Ine}\, (D)$. As $g$ preserves each curves $C_i$ and $C_i$ ($1 \le i \le 8$) generate ${\rm NS}\, (S)$, it follows that $g^* = id$ on ${\rm NS}\, (S)$. On the othe hand, as ${\rm NS}\, (S)$ is negative definite of rank $8$, the transcendental lattice $T(S)$ is of rank $14$, which is preserved by $g$. Then by Theorem \ref{thm51} (2), the characteristic polynomial of $g^*|T(S)$ is exactly $\varphi_{14}(x)$. As $\varphi_{14}(x)$ is irreducible over $\Z$ and $g$ preserves $\C \omega_S \subset T(S)_{\C}$, it follows that $g^*\omega_S = \delta \omega_S$ for some Galois conjugate $\delta$ of $\alpha$. Note that $\delta$ is {\it not} a root of unity, as $\delta$ is a Galois conjugate of $\alpha >1$, while $|\delta| =1$ by $(f^*\omega_S, f^*\overline{\omega_S}) = (\omega_S, \overline{\omega_S}) \not= 0$. Here we denote by $\overline{*}$ the complex conjugate of $*$ and $|\delta| := \sqrt{\delta \overline{\delta}} \in \R_{\ge 0}$. 

Let $P$ be the intersection point of $C = C_3$ and $D = C_8$. We denote by $x$ an affine coordinate of $D \simeq \P^1$ with $x(P) = 0$. As $g|_{C} = id_C$, $C$ and $D$ meets at $P$ transversally and $g^{*}\omega_S = \delta \omega_S$, it follows that $(g|_{D})^* x = \delta x$, i.e., $g|_D \in {\rm Aut}\, (D)$ is the multiplication automorphism $\delta \cdot$ by $\delta$. As ${\rm Aut}\, (S) = \langle g \rangle$ by Theorem \ref{thm51} (1) and $\delta$ is not a root of unity, it follows that the natural restriction homomorphism 
$$\tau : \langle g \rangle = {\rm Aut}\, (S) = {\rm Dec}\, (D) \to {\rm Aut}\, (D)\,\, ;\,\, g^n \mapsto (\delta \cdot)^n = \delta^n \cdot$$
is injective. Hence ${\rm Ine}\, (D) = {\rm Ker}\, \tau = \{id_S\}$. This completes the proof of Theorem \ref{thm3}.

\end{document}